\documentclass[12pt,oneside,reqno]{amsart}
\usepackage[all]{xy}
\usepackage{amsfonts,amsmath,oldgerm,amssymb,amscd,comment,multirow,mathrsfs}
\UseComputerModernTips
\numberwithin{equation}{section}

\usepackage[breaklinks]{hyperref}
\usepackage{float}        
\usepackage{placeins}     

\allowdisplaybreaks

\usepackage{enumerate}
\usepackage{caption} 
\newtheorem{theorem}{Theorem}[section]
\newtheorem{prop}{Proposition}

\newtheorem{lemma}[theorem]{{\bf Lemma}}
\newtheorem{coro}[theorem]{{\bf Corollary}}

\newtheorem{remark}[subsection]{Remark}

\begin{document}
	\title[Moments of the Crank Statistic for 
	$t$-Core Partitions and Overpartitions]  
	{Moments of the Crank Statistic for 
		$t$-Core Partitions and Overpartitions} 
	
	\author[ Sourav Bhowmick]{S. Bhowmick}
	\address{Sourav Bhowmick, Department of Mathematics, National Institute of Technology, Raipur, Chhattisgarh 492010.}
	\email{souravbhowmick578@gmail.com, sbhowmick.phd2025.maths@nitrr.ac.in}
	
	\author[N.K. Meher]{N.K. Meher}
	\address{Nabin Kumar Meher, Department of Mathematics, National Institute of Technology, Raipur, Chhattisgarh 492010.}
	\email{mehernabin@gmail.com, nkmeher.maths@nitrr.ac.in}

	\maketitle
	\pagenumbering{arabic}
		\begin{abstract}
			Recently, Kang, Kim, and Lee \cite{Kang2026} developed a unified moment-trace framework for symmetric partition statistics using complete Bell polynomials and their inversion formula. In this paper, we apply this framework to crank statistics for $t$-core partitions and overpartitions. For $t\in\{5,7,11,17,19\}$, we show that the normalized even crank moment generating functions for $t$-core partitions admit partition-trace representations in terms of the functions $D^{(t)}_{2s}(\tau)$, together with suitable Bernoulli-number shifts. We also establish inverse trace formulas that recover $D^{(t)}_{2s}(\tau)$ from the corresponding normalized even crank moments. For overpartitions, we obtain analogous trace and inverse-trace identities for the normalized even moments associated with the first and second residual crank generating functions. As applications, we use complete Bell polynomials and their inversion formula to obtain explicit expressions for the $t$-core partition numbers and overpartitions number in terms of sums involving divisor function.
		
	\end{abstract}
	
	\maketitle
	\section{Introduction}
	A partition of a positive integer $n$ is a finite nonincreasing sequence 
	of positive integers 
	\[ 
	\beta_1\geq\beta_2\geq\cdots\geq\beta_s 
	\] 
	such that 
	\[ 
	\sum_{i=1}^{s}\beta_i=n. 
	\] 
The partition function has been studied through various statistics in 
order to understand its arithmetic as well as probabilistic behavior. 
Among the most important examples are the \emph{crank} and the 
\emph{rank}. Dyson \cite{Dyson1944} introduced the rank as a statistic that 
provides a combinatorial interpretation of Ramanujan's congruences for 
the partition function. For a partition $\beta$, its rank is defined 
by 
\[ 
\operatorname{rank}(\beta):=\beta_1-\ell(\beta), 
\] 
where $\beta_1$ represents the largest part of $\beta$, while 
$\ell(\beta)$ denotes the total number of parts of $\beta$. 

The crank was conjectured by Dyson and subsequently introduced in its 
present form by Andrews and Garvan \cite{AndrewsGarvan1988}. Let $\mu^*(\beta)$ denote the count of parts equal to $1$, and let $\nu^*(\beta)$ denote the count of parts strictly exceeding
that value. The crank is then defined piecewise: it equals the largest part
$\beta_1$ whenever $\beta$ contains no $1$'s ($\mu^*(\beta) = 0$), and
otherwise equals the difference $\nu^*(\beta) - \mu^*(\beta)$.
If $\beta=(\beta_1,\beta_2,\ldots,\beta_k)$ is a partition of 
$n$, then the corresponding Ferrers-Young diagram of $\beta$ is the 
following stair-step arrangement of nodes with $\beta_i$ nodes in the 
$i$-th row: 
\[ 
\begin{array}{ll} 
	\bullet\quad\bullet\quad\cdots\quad\bullet\quad\bullet 
	& \beta_1\text{ nodes}\\[5pt] 
	\bullet\quad\bullet\quad\cdots\quad\bullet 
	& \beta_2\text{ nodes}\\[5pt] 
	\vdots & \\[3pt] 
	\bullet\quad\cdots\quad\bullet 
	& \beta_k\text{ nodes}. 
\end{array} 
\] 
Label the nodes $(i,j)$ as the entries of a matrix. The $(i,j)$-hook is 
the set consisting of the node itself, the nodes directly to its right, 
and the nodes directly below it. Let $\beta'_j$ denote the number of 
nodes in the $j$-th column. The hook number $H(i,j)$ of the node $(i,j)$ 
is defined by 
\[ 
H(i,j)=\beta_i+\beta'_j-i-j+1. 
\] 
A $t$-core partition of $n$ is a partition of $n$ for which none of the 
hook numbers is divisible by $t$. We illustrate the Ferrers--Young diagram of the partition 
$5+4+3+1$ of $13$, together with its hook numbers, as follows: 
\[ 
\begin{array}{ccccc} 
	\bullet\,8 & \bullet\,6 & \bullet\,5 & \bullet\,3 & \bullet\,1\\[4pt] 
	\bullet\,6 & \bullet\,4 & \bullet\,3 & \bullet\,1 & \\[4pt] 
	\bullet\,4 & \bullet\,2 & \bullet\,1 & & \\[4pt] 
	\bullet\,1 & & & & 
\end{array} 
\] 
Since none of the hook numbers is divisible by $7$, the partition 
$5+4+3+1$ of $13$ is a $7$-core partition. Moreover, since the 
largest hook number is $8$, it is a $t$-core partition for every 
$t\geq 9$.

An overpartition of $n$ is a partition of $n$ in which the first
occurrence of each part may be overlined. For example, there are $24$
overpartitions for $n=5$, given below in tabular format.

\begin{table}[H]
	\centering
	\caption{Overpartitions of $5$}
	\label{table1}
	\renewcommand{\arraystretch}{1.3}
	\begin{tabular}{|p{0.95\textwidth}|}
		\hline
		
		$5,\ \overline{5}$ \\ \hline
		
		$4+1,\ \overline{4}+1,\ 4+\overline{1},\
		\overline{4}+\overline{1}$ \\ \hline
		
		$3+2,\ \overline{3}+2,\ 3+\overline{2},\
		\overline{3}+\overline{2}$ \\ \hline
		
		$3+1^2,\ \overline{3}+1^2,\ 3+\overline{1}+1,\
		\overline{3}+\overline{1}+1$ \\ \hline
		
		$2^2+1,\ \overline{2}+2+1,\ 2^2+\overline{1},\
		\overline{2}+2+\overline{1}$ \\ \hline
		
		$2+1^3,\ \overline{2}+1^3,\ 2+\overline{1}+1^2,\
		\overline{2}+\overline{1}+1^2$ \\ \hline
		
		$1^5,\ \overline{1}+1^4$ \\ \hline
		
	\end{tabular}
\end{table}
Let $p(n)$ denote the number of partitions of $n$. The generating
function for $p(n)$ is given by
\begin{equation}\label{eq1.1}
\sum_{n=0}^{\infty}p(n)q^n
=
\frac{1}{(q;q)_\infty}
=
\frac{1}{f_1},
\end{equation}
where, 
\[
f_j:=(q^j;q^j)_\infty
=\prod_{n=1}^{\infty}(1-q^{nj}).
\]
For $t\geq 1$, let $b_t(n)$ denote the number of $t$-core partitions of
$n$. Garvan et al.\cite{GarvanKimStanton1990} showed that
\begin{equation}\label{eq1.2}
	\sum_{n=0}^{\infty}b_t(n)q^n
	=
	\frac{f_t^t}{f_1}.
\end{equation}.\\ For more arithmetic properties of  $t$-core partitions one can read the paper of \cite{JindalMeher2024}.
 
	Corteel and Lovejoy \cite{Lovejoy2004} obtained the following generating function for the overpartition function $\bar{p}(n)$:
	\begin{align}\label{eq1.3}
		\sum_{n\geq0}\bar{p}(n)q^n=\frac{f_2}{f_1^2}.
	\end{align}

For integers $m$ and non-negative integers $n$, denote by $N(m,n)$ and $M(m,n)$ the numbers of partitions of $n$ whose rank and crank, respectively, are equal to $m$. These two statistics give rise to natural two-variable generating functions that record their distributions simultaneously. In particular, the crank generating function, as established by Andrews and Garvan \cite{AndrewsGarvan1988}, is
\begin{equation}\label{eq1.4}
	C(\zeta;q)
	:=\sum_{n \geq 0} \sum_{m \in \mathbb{Z}}
	M(m,n)\zeta^m q^n
	=
	\frac{(q;q)_\infty}
	{(\zeta q;q)_\infty(\zeta^{-1}q;q)_\infty}.
\end{equation}.
 The generating function for the rank is given in \cite{AtkinSwinnertonDyer1954},
\begin{equation}\label{eq1.5}
R(\zeta;q)
:=\sum_{n \geq 0} \sum_{m \in \mathbb{Z}}
N(m,n)\zeta^m q^n
=
\sum_{n\geq0}
\frac{q^{n^2}}
{(\zeta q;q)_n(q/\zeta;q)_n}.
\end{equation}

Suppose $u(m,n)$ denotes the total number of weakly unimodal sequences with rank $m$ and weight $n$. We then obtain the two-variable generating function as \cite{KimLovejoy2014}
\begin{equation}\label{eq1.6}
U(\zeta; q) := \sum_{n \geq 0} \sum_{m \in \mathbb{Z}} u(m, n) \zeta^m q^n =
\sum_{n \geq 0} \frac{q^n}{(\zeta q;q)_n (q/\zeta;q)_n}.
\end{equation}

Taking $z$ and $\tau$ to be the elliptic and modular variables
respectively, and letting $\zeta = e^{2\pi i z}$, $q = e^{2\pi i
\tau}$, $R(\zeta;q)$ is only a mock Jacobi form of index $-\tfrac{3}{2}$ and weight $\tfrac{1}{2}$, whereas $C(\zeta;q)$ turns out to be a meromorphic Jacobi form of index $-\tfrac{1}{2}$ and weight $\tfrac{1}{2}$.
 
For the rank, crank and unimodal sequence statistics, the associated
$r$-th moment generating functions are given as follows:

\begin{equation}\label{eq1.8}
	R_r(q):=
	\sum_{n\geq0}\sum_{m\in\mathbb{Z}}
	m^rN(m,n)q^n,
\end{equation}
\begin{equation}\label{eq1.7}
	C_{r}(q):=
	\sum_{n\ge0}\sum_{m\in\mathbb{Z}}
	m^{r}M(m,n)q^{n},
\end{equation}
and
\begin{equation}\label{eq1.9}
	U_r(q):=
	\sum_{n\geq0}\sum_{m\in\mathbb{Z}}
	m^ru(m,n)q^n.
\end{equation}
Alternatively, one may derive the $r$-th moment generating functions
directly from the two-variable generating functions via the operator
$\zeta \frac{d}{d\zeta}$, evaluated at $\zeta = 1$:
\begin{equation}\label{eq1.10}
	C_r(q)=\left.
	\left(\zeta\frac{d}{d\zeta}\right)^r C(\zeta;q)
	\right|_{\zeta=1},
	\end{equation}
\begin{equation}\label{eq1.11}
  R_r(q)=\left.
	\left(\zeta\frac{d}{d\zeta}\right)^r R(\zeta;q)
	\right|_{\zeta=1},
	\end{equation}
and
\begin{equation}\label{eq1.12}
	U_r(q)=\left.
	\left(\zeta\frac{d}{d\zeta}\right)^r
	U(\zeta;q)
	\right|_{\zeta=1}.
\end{equation}

For any even integer $s \geq 2$, we define the normalized Eisenstein
series of weight $s$ as
\[
E_s(\tau)
=
1-\frac{2s}{B_s}
\sum_{n\geq1}\sigma_{s-1}(n)q^n,
\]
in which $B_s$ stands for the $s$-th Bernoulli number. When
$s \geq 4$, $E_s$ is known to be a modular form over
$\mathrm{SL}_2(\mathbb{Z})$. However, the situation at $s=2$ differs:
despite possessing an analogous $q$-expansion, $E_2$ does not obey
the standard modular transformation law, placing it instead in the
category of quasimodular forms.

These Eisenstein series are also connected with partition statistics.
In particular, recent work of Amdeberhan et al.\cite{Amdeberhan2025} gives explicit formula for the even crank moments through partition Eisenstein traces. For this purpose, it is convenient to express a partition using its frequency notation:
\[
\beta=(1^{m_1},2^{m_2},\ldots,s^{m_s})\vdash s,
\]
where, $m_k$ denotes the multiplicity of the part $k$ in $\beta$.
 Consequently, the length of the partition is
\[
\ell(\beta)=m_1+m_2+\cdots+m_s.
\]
For a partition
\[
\beta=(1^{m_1},2^{m_2},\ldots,s^{m_s}),
\]
let us attach to $\beta$ the monomial
\[
F_\beta:=\prod_{k=1}^{s}F_k^{m_k}.
\]
Now let $\mathscr{A}$ denote the collection of all partitions and let
$\Psi:\mathscr{A}\to\mathbb{C}$ be any function. The corresponding
partition trace is defined by summing the values of $\Psi$ over all
partitions of $s$, each weighted by its associated monomial:
\begin{equation}\label{eq1.13}
	\operatorname{Tr}_{s}(\Psi;F_1,\ldots,F_s)
	:=
	\sum_{\beta\vdash s}\Psi(\beta)F_\beta.
\end{equation}
A particularly important family arises when the function $\Psi$ is
chosen from the Eisenstein series. More precisely, the partition Eisenstein traces are obtained from the sequence
$F=\{G_s\}_{s\geq1}$,
where
\begin{equation}\label{eq1.14}
	G_{2s}(\tau)
	:=
	-\frac{B_{2s}}{4s}E_{2s}(\tau)
	=
	-\frac{B_{2s}}{4s}
	+\sum_{n\geq1}\sigma_{2s-1}(n)q^n,\\
\end{equation}
and
\[
G_{2s-1}(\tau)=0.
\]
We next formulate the crank moment generating function through traces of partition Eisenstein series. The following formulation is adapted from \cite[Theorem~1.1]{Bringmann2025} and is equivalent to the result established in \cite[Theorem~1.2]{Amdeberhan2025}.

\begin{theorem}\label{thm1}
	For the crank moment generating functions, we have
	\begin{equation}\label{eq1.15}
		\sum_{s\geq0}C_s(q)\frac{z^s}{s!}
		=
		\frac{2\sinh(z/2)}{z(q)_\infty}
		\sum_{s\geq0}
		\operatorname{Tr}_s
		(\Psi;G_1,G_2,\ldots,G_s)z^s,
	\end{equation}
	where
	\[
	\Psi(\beta)
	:=
	\prod_{k\geq1}
	\frac{2^{m_j}}{m_j!\,j^{m_j}},
	\qquad
	\beta=(1^{m_1},2^{m_2},\ldots,s^{m_s})\vdash s.
	\]
\end{theorem}

Recenly, Bringmann, Pandey, and van
Ittersum \cite{Bringmann2025} obtained the following expansion in terms
of the Eisenstein series:
\begin{equation}\label{eq1.16}
	C(\zeta;q)
	=
	\frac{\sin(\pi z)}{\pi z(q;q)_\infty}
	\exp\left(
	2\sum_{s\geq2}
	G_s(\tau)\frac{(2\pi iz)^s}{s!}
	\right).
\end{equation}

In a similar spirit, they \cite{Bringmann2025} introduced a family $f_s$ of mock Eisenstein series by writing the rank generating function in the exponential form
\begin{equation}\label{eq1.17}
	R(\zeta;q)
	=
	\frac{\sin(\pi z)}{\pi z(q;q)_\infty}
	\exp\left(
	2\sum_{s\geq1}
	f_s(\tau)\frac{(2\pi iz)^s}{s!}
	\right),
\end{equation}
and for the rank moments they established the following 
\begin{theorem}\label{thm2}
	The rank moment generating functions satisfy
	\begin{equation}\label{eq1.18}
		\sum_{s\geq0}R_s(q)\frac{z^s}{s!}
		=
		\frac{2\sinh(z/2)}{z(q;q)_\infty}
		\sum_{s\geq0}
		\operatorname{Tr}_s
		(\Psi;f_1,f_2,\ldots,f_s)z^s.
	\end{equation}
\end{theorem}

The connection between partition traces and generating functions can
be viewed from a broader perspective. Matsusaka \cite{Matsusaka2025} recently
demonstrated that a wide range of generating functions arising in
partition theory can be represented systematically through complete
Bell polynomials, or, in an equivalent formulation, through partition
traces. 

Inspired by Matsusaka \cite{Matsusaka2025} more recently, Kang, Kim, and Lee \cite{Kang2026} developed a unified framework for studying the moment generating functions associated with combinatorial statistics by means of complete Bell polynomials and partition traces. Their approach combines an algebraic method based on the inversion formula for complete Bell polynomials with an analytic method arising from Faà di Bruno's formula. In particular, for a broad class of symmetric combinatorial statistics, they established explicit relations between the even moments and the corresponding Eisenstein-type functions, and conversely recovered these functions from the moment generating functions through partition traces.

For the crank case, Kang et al.\cite{Kang2026} proved the following result.
\begin{coro}
	For every integer $s\geq1$, define
	\[
	\mathcal{C}_s(q)=(q)_\infty C_s(q).
	\]
	Then the normalized even crank moments satisfy
	\begin{equation}\label{eq1.28}
		\mathcal{C}_{2s}(q)
		=
		\operatorname{Tr}_s
		\left(
		\Psi_T;
		G_2+\frac{B_2}{4},
		G_4+\frac{B_4}{8},
		\dots,
		G_{2s}+\frac{B_{2s}}{4s}
		\right),
	\end{equation}
	and the reverse relation is
	\begin{equation}\label{eq1.29}
		G_{2s}(\tau)
		=
		-\frac{B_{2s}}{4s}
		+
		\operatorname{Tr}_s
		\left(
		\Psi_V;
		\mathcal{C}_2(q),
		\mathcal{C}_4(q),
		\dots,
		\mathcal{C}_{2s}(q)
		\right),
	\end{equation}
	where, for $\beta=(1^{m_1},2^{m_2},\ldots,s^{m_s})\vdash s$, we have 
	\begin{equation*} 
		\Psi_T(\beta) 
		= 
		(2s)! 
		\prod_{k=1}^{s} 
		\frac{2^{m_k}} 
		{m_j!\,((2k)!)^{m_k}}, 
	\end{equation*} 
	and 
	\begin{equation*} 
		\Psi_V(\beta) 
		= 
		\frac{(2s)!}{2} 
		(-1)^{\ell(\beta)-1} 
		(\ell(\beta)-1)! 
		\prod_{k=1}^{s} 
		\frac{1} 
		{m_k!\,((2k)!)^{m_k}}. 
	\end{equation*} 
\end{coro}

For the rank statistic, they obtained the following-
\begin{coro}
	For each integer $s\geq1$, put
	\[
	\mathcal{R}_s(q)=(q)_\infty R_s(q).
	\]
	Then
	\begin{equation}\label{eq1.30}
		\mathcal{R}_{2s}(q)
		=
		\operatorname{Tr}_s
		\left(
		\Psi_T;
		f_2+\frac{B_2}{4},
		f_4+\frac{B_4}{8},
		\dots,
		f_{2s}+\frac{B_{2s}}{4s}
		\right),
	\end{equation}
	while the corresponding inverse identity is
	\begin{equation}\label{eq1.31}
		f_{2s}(q)
		=
		-\frac{B_{2s}}{4s}
		+
		\operatorname{Tr}_s
		\left(
		\Psi_V;
		\mathcal{R}_2(q),
		\mathcal{R}_4(q),
		\dots,
		\mathcal{R}_{2s}(q)
		\right),
	\end{equation}
	 where $\Psi_T$ and $\Psi_V$ are same as above.
\end{coro}

A related application concerns the rank statistic for unimodal
sequences. The unimodal rank moments are connected with
Eisenstein-type series known as false Eisenstein series and
partial Eisenstein series. These functions were introduced in
\cite{Bringmann2026}. In this setting, let $u_s(\tau)$ be defined
through the expansion
\begin{equation}\label{eq1.32}
	U(\zeta;q)
	=
	\frac{\sin(\pi z)}{\pi z}U(1;q)
	\exp\left(
	2\sum_{s\geq1}
	u_s(\tau)\frac{(2\pi iz)^s}{s!}
	\right).
\end{equation}

\begin{coro}[{\cite[Corollary 4.4]{Kang2026}}]
	For integers $s,j\geq1$, define
	\[
	\mathcal{U}_{2s}(q)
	:=
	\frac{U_{2s}(q)}{U(1;q)}
	\]
	and set
	\[
	W_{2j}(\tau)
	:=
	u_{2j}(\tau)+\frac{B_{2j}}{4j}.
	\]
	Then the normalized even moments are represented by
	\begin{equation}\label{eq1.33}
		\mathcal{U}_{2s}(q)
		=
		\operatorname{Tr}_{s}
		\bigl(
		\Psi_T;
		W_2,W_4,\ldots,W_{2s}
		\bigr),
	\end{equation}
	and the associated Eisenstein-type functions can be recovered from
	the moments through
	\begin{equation}\label{eq1.34}
		W_{2s}(\tau)
		=
		\operatorname{Tr}_{s}
		\bigl(
		\Psi_V;
		\mathcal{U}_2(q),\mathcal{U}_4(q),\ldots,\mathcal{U}_{2s}(q)
		\bigr).
	\end{equation}
\end{coro}
In this paper, motivated by Kang, Kim, and Lee \cite{Kang2026}, we discuss crank statistics of $t$-core partitions with $t\in\{5,7,11,17,19\}$, and overpartitions. We show that the corresponding normalized even crank moment generating functions can be written as partition traces of $D^{(t)}_{2s}(\tau)$, shifted by Bernoulli numbers, and we further establish the inverse relations recovering these series from the crank moments.

The main results of this paper can be stated as follows.

For $t$-core partitions with $t\in\{5,7,11,17,19\}$, we have the following theorem.
	\begin{theorem}\label{th1} 
	Let $\mathcal{C}^{(t)}(\zeta;q)$ be the two-variable crank generating 
	function for $t$-core partitions, and let $s$ be a positive integer. 
	Let $\mathcal{C}^{(t)}_{2s}(q)$ denote the normalized $2s$-th crank 
	moment generating function. Then 
	\begin{equation}\label{eq:5} 
		\begin{aligned} 
			\mathcal{C}_{2s}^{(t)}(q) 
			&= 
			\operatorname{Tr}_s 
			\left( 
			\Psi_T; 
			D^{(t)}_{2}+\frac{B_2}{4}, 
			D^{(t)}_{4}+\frac{B_4}{8}, 
			\ldots, 
			D^{(t)}_{2s}+\frac{B_{2s}}{4s} 
			\right),\\ 
			D^{(t)}_{2s} 
			&= 
			-\frac{B_{2s}}{4s} 
			+ 
			\operatorname{Tr}_s 
			\left( 
			\Psi_V; 
			\mathcal{C}_2^{(t)}(q), 
			\mathcal{C}_4^{(t)}(q), 
			\ldots, 
			\mathcal{C}_{2s}^{(t)}(q) 
			\right). 
		\end{aligned} 
	\end{equation} 
	where 
	\begin{equation*} 
		D^{(t)}_{s}(\tau) 
		= 
		\begin{cases} 
			A^{(t)}_{s}(\tau)+G_s(\tau), 
			& \text{if $s$ is even},\\[4pt] 
			0, 
			& \text{if $s$ is odd}. 
		\end{cases} 
	\end{equation*} 
	and 
	\begin{equation}\label{eq:At} 
		A^{(t)}_{2s}(\tau) 
		= 
		-\left(\sum_{r\in\mathcal{R}_t}r^{2s}\right) 
		\sum_{m=1}^{\infty} 
		m^{2s-1}\frac{q^{tm}}{1-q^{tm}}, 
	\end{equation} 
	and
	\[ 
	\mathcal{R}_t 
	= 
	\begin{cases} 
		\{2,4\}, & t=5,7,\\[2pt] 
		\{2,4,8,10\}, & t=11,\\[2pt] 
		\{2,4,8,10,14,16\}, & t=17,19. 
	\end{cases} 
	\] 
	and for $\beta=(1^{m_1},2^{m_2},\ldots,s^{m_s})\vdash s$, we have 
	\begin{equation*} 
		\Psi_T(\beta) 
		= 
		(2s)! 
		\prod_{k=1}^{s} 
		\frac{2^{m_k}} 
		{m_j!\,((2k)!)^{m_k}}, 
	\end{equation*} 
	and 
	\begin{equation*} 
		\Psi_V(\beta) 
		= 
		\frac{(2s)!}{2} 
		(-1)^{\ell(\beta)-1} 
		(\ell(\beta)-1)! 
		\prod_{k=1}^{s} 
		\frac{1} 
		{m_k!\,((2k)!)^{m_k}}. 
	\end{equation*} 
\end{theorem} 
For the first residual crank and second residual crank of overpartition, we prove the following theorems respectively. 
\begin{theorem}\label{th2} 
	Let $\overline{C}(\zeta;q)$ be the first residual crank generating function for overpartitions, and $s$ be a natural number. Let $\overline{\mathcal{C}}_{2s}(q)$ denote the normalized $2s$-th crank moment generating function. Then 
	
	\[ 
	\overline{\mathcal{C}}_{2s}(q) 
	= 
	\operatorname{Tr}_s 
	\left( 
	\Psi_T; 
	\, 
	G_2+\frac{B_2}{4}, 
	\, 
	G_4+\frac{B_4}{8}, 
	\, 
	\ldots, 
	\, 
	G_{2s}+\frac{B_{2s}}{4s} 
	\right), 
	\] 
	and 
	\[ 
	G_{2s} 
	= 
	-\frac{B_{2s}}{4s} 
	+ 
	\operatorname{Tr}_s 
	\left( 
	\Psi_V; 
	\, 
	\overline{\mathcal{C}}_{2}(q), 
	\, 
	\overline{\mathcal{C}}_{4}(q), 
	\, 
	\ldots, 
	\, 
	\overline{\mathcal{C}}_{2s}(q) 
	\right), 
	\]  
	where $\Psi_T(\beta)$ and $\Psi_V(\beta)$ are the same as those defined in the foregoing theorem. 
\end{theorem} 

\begin{theorem}\label{th3} 
	Let $\overline{C2}(\zeta;q)$ be the second residual crank generating function for overpartitions. For every integer $s\geq1$, let $\overline{C2}_{2s}(q)$ be the normalized even moments associated with	$\overline{C2}(\zeta;q)$. Then 
	\begin{equation*} 
		\overline{C2}_{2s}(q) 
		= 
		\operatorname{Tr}_s 
		\left( 
		\Psi_T; 
		2G_2(2\tau)+\frac{B_2}{4}, 
		2G_4(2\tau)+\frac{B_4}{8}, 
		\ldots, 
		2G_{2s}(2\tau)+\frac{B_{2s}}{4s} 
		\right), 
	\end{equation*} 
	Moreover, 
	\begin{equation*} 
		G_{2s}(2\tau) 
		= 
		-\frac{B_{2s}}{4s} 
		+ 
		\operatorname{Tr}_s 
		\left( 
		\Psi_V; 
		\overline{C2}_{2}(q), 
		\overline{C2}_{4}(q), 
		\ldots, 
		\overline{C2}_{2s}(q) 
		\right), 
	\end{equation*} 
	where $\Psi_T(\beta)$ and $\Psi_V(\beta)$ are the same as earlier. 
\end{theorem}

	\section{Preliminaries}
	 We begin by reviewing the complete Bell polynomials together with their corresponding inversion formula. For more details about Bell polynomials one can see the paper of T. Matsusaka \cite{Matsusaka2025}, and Kang et al.\cite{Kang2026}.
	\subsection{Bell Polynomials and M\"obius Inversion}
	We will use the concept of set partitions. Let $R$ be a finite set.
	A set partition of $R$ is a family
	\[
	\rho=\{B_1,B_2,\ldots,B_h\}
	\]
	of nonempty, pairwise disjoint subsets $B_i\subseteq R$ satisfying
	\[
	B_1\cup B_2\cup\cdots\cup B_h=R.
	\]
	The subsets $B_i$ are called the blocks of $\rho$, and we set
	\[
	|\rho|:=h,
	\]
	which represents the number of blocks in $\rho$. Let $\mathscr{D}_s$
	denote the set of all set partitions of
	\[
	[s]:=\{1,2,\ldots,s\}.
	\]
	
	For $\rho\in\mathscr{D}_s$, let
	\[
	\beta(\rho):=(1^{m_1},2^{m_2},\ldots,s^{m_s})\vdash s,
	\]
	where $m_j$ denotes the number of blocks of $\rho$ having size $j$.
	It follows that
	\[
	|\rho|=m_1+m_2+\cdots+m_s=\ell(\beta).
	\]
	The Bell number $B(s)$ counts the set partitions of a set containing
	$s$ elements. Its exponential generating function is
	\[
	\sum_{s=0}^{\infty}B(s)\frac{x^s}{s!}
	=
	\exp(e^x-1).
	\]
	we introduce the complete Bell polynomial
	\[
	B_s(Y_1,\ldots,Y_s)\in\mathbb{Z}[Y_1,\ldots,Y_s],
	\]
	which is defined by the generating function
	\begin{equation}\label{2e1}
		\sum_{s=0}^{\infty} B_s(Y_1,\ldots,Y_s)\frac{t^s}{s!}
		=
		\exp\left(\sum_{j=1}^{\infty}Y_j\frac{t^j}{j!}\right).
	\end{equation}
	From the definition, or equivalently from its exponential generating
	function, we have
	\begin{equation}\label{2e2}
		B_s(Y_1,\ldots,Y_s)
		=
		\sum_{\beta\vdash s}
		\Psi_B(\beta)Y_\beta,
	\end{equation}
	where, for
	\[
	\beta=(1^{m_1},\ldots,s^{m_s})\vdash s,
	\]
	we define
	\begin{equation}\label{2e3}
		\Psi_B(\beta)
		=
		s!\prod_{j=1}^{s}
		\frac{1}{m_j!(j!)^{m_j}}.
	\end{equation}
	We next recall the inversion formula for the complete Bell polynomials. Let
	$\{J_s\}_{s\geq 1}$ and $\{F_s\}_{s\geq 1}$ be two sequences satisfying
	\begin{equation}\label{2e4}
		1+\sum_{s\geq 1}F_s\frac{t^s}{s!}
		=
		\exp\left(\sum_{s\geq 1}J_s\frac{t^s}{s!}\right).
	\end{equation}
	Thus, the sequence $\{F_s\}_{s\geq 1}$ is determined by the complete Bell
	polynomials according to
	\begin{equation}\label{2e5}
		F_s
		=
		B_s(J_1,\ldots,J_s)
		=
		\sum_{\beta\vdash s}
		\Psi_B(\beta)J_\beta.
	\end{equation}
	Taking the logarithm of both sides, we obtain
	\begin{equation}\label{2e6}
		\sum_{s\geq 1}J_s\frac{t^s}{s!}
		=
		\log\left(1+\sum_{s\geq 1}F_s\frac{t^s}{s!}\right).
	\end{equation}
Consequently, the coefficients $J_s$ can be recovered from the sequence
$\{F_s\}_{s\geq 1}$ through
\begin{equation}\label{2e7}
	J_s
	=
	\sum_{\rho\in\mathscr{D}_s}
	(-1)^{|\rho|-1}(|\rho|-1)!
	\prod_{B\in\rho}F_{|B|}.
\end{equation}
In terms of integer partitions, the preceding relation takes the form
\begin{equation}\label{2e8}
	J_s
	=
	\sum_{\beta\vdash s}
	\mu(\beta)\Psi_B(\beta)F_\beta,
\end{equation}
where
\begin{equation}\label{2e9}
	\mu(\beta)
	=
	(-1)^{\ell(\beta)-1}
	\bigl(\ell(\beta)-1\bigr)!.
\end{equation}

 Now, to recall the main results of Kang et al. \cite{Kang2026}, we first introduce the necessary preliminaries.
	Let $w(\lambda)$ be an integer-valued statistic on combinatorial
	objects. For each $n$, let $T(n)$ count the objects of weight $n$,
	while $T(m,n)$ denotes the number of objects $\lambda$ of weight $n$
	for which $w(\lambda)=m$. In addition, we impose the following
	symmetry condition on $w(\lambda)$:
	\begin{equation}\label{eq1.19}
		T(m,n)=T(-m,n).
	\end{equation}
	The associated two-variable generating function is given by
	\begin{equation}\label{eq1.20}
		\Phi_T(\zeta;q)
		=
		\sum_{n\geq0}\sum_{m\in\mathbb Z}
		T(m,n)\zeta^m q^n.
	\end{equation}
	Assume further that this generating function possesses the
	exponential representation
	\begin{equation}\label{eq1.21}
		\Phi_T(\zeta;q)
		=
		\frac{\sin(\pi z)}{\pi z}\Phi_T(1;q)
		\exp\left(
		2\sum_{j\geq2}
		\Theta_j(\tau)\frac{(2\pi iz)^j}{j!}
		\right).
	\end{equation}
	For $r\geq0$, we introduce the $r$-th moment generating function
	\begin{equation}\label{eq1.22}
		\mu_r(q)
		=
		\sum_{n\geq0}
		\left(
		\sum_{m\in\mathbb Z}m^rT(m,n)
		\right)q^n,
	\end{equation}
	and its normalized r-th moment by
	\begin{equation}\label{eq1.23}
		\widehat{\mu}_r(q)
		=
		\frac{\mu_r(q)}{\Phi_T(1;q)}.
	\end{equation}
	
	\begin{lemma}[{\cite[Theorem 1.3]{Kang2026}}]\label{KANG}
		For each positive integer $s$, suppose that
		$\widehat{\mu}_{2s}(q)$ is the normalized $2s$-th moment generating
		function corresponding to a statistic satisfying the symmetry
		\eqref{eq1.19}. Then
		\begin{equation}\label{eq1.24}
			\widehat{\mu}_{2s}(q)
			=
			\operatorname{Tr}_s\left(
			\Psi_T;
			\Theta_2+\frac{B_2}{4},
			\Theta_4+\frac{B_4}{8},
			\ldots,
			\Theta_{2s}+\frac{B_{2s}}{4s}
			\right).
		\end{equation}
		Moreover, the functions $\Theta_{2s}$ satisfy the inverse relation
		\begin{equation}\label{eq1.25}
			\Theta_{2s}
			=
			-\frac{B_{2s}}{4s}
			+
			\operatorname{Tr}_s
			\left(
			\Psi_V;
			\widehat{\mu}_2(q),\widehat{\mu}_4(q),
			\ldots,\widehat{\mu}_{2s}(q)
			\right),
		\end{equation}
		where for $\beta=(1^{m_1},2^{m_2},\ldots,s^{m_s})\vdash s$, we have 
		\begin{equation*} 
			\Psi_T(\beta) 
			= 
			(2s)! 
			\prod_{k=1}^{s} 
			\frac{2^{m_k}} 
			{m_j!\,((2k)!)^{m_k}}, 
		\end{equation*} 
		and 
		\begin{equation*} 
			\Psi_V(\beta) 
			= 
			\frac{(2s)!}{2} 
			(-1)^{\ell(\beta)-1} 
			(\ell(\beta)-1)! 
			\prod_{k=1}^{s} 
			\frac{1} 
			{m_k!\,((2k)!)^{m_k}}. 
		\end{equation*} 
	\end{lemma}

	In analogy with the crank statistic for ordinary partitions, Bringmann et al. \cite{BringmannO2009} introduced the first and second residual cranks for overpartitions. The first residual crank is the crank of the subpartition formed by the nonoverlined parts, while the second residual crank is the crank of the even nonoverlined parts after dividing them by two.
	\begin{lemma}
	 Let $\overline{M}(m;n)$ and $\overline{M2}(m;n)$ denote the number of overpartitions of $n$ having first residual crank and second residual crank $m$ respectively. Then
	\begin{equation}\label{eq:1.2}
	\overline{C}(\zeta;q)
	:=
	\sum_{n=0}^{\infty}
	\sum_{m=-\infty}^{\infty}
	\overline{M}(m;n)\zeta^{m}q^{n}
	=
	\frac{(q^{2};q^{2})_{\infty}}
	{(\zeta q;q)_{\infty}(q/\zeta;q)_{\infty}},
	\end{equation}
	and
	\begin{equation}\label{eq:1.3}
	\overline{C2}(\zeta;q)
	:=
	\sum_{n=0}^{\infty}
	\sum_{m=-\infty}^{\infty}
	\overline{M2}(m;n)\zeta^{m}q^{n}
	=
	\frac{(-q;q)_{\infty}(q^{2};q^{2})_{\infty}}
	{(q;q^{2})_{\infty}
		(\zeta q^{2};q^{2})_{\infty}
		(q^{2}/\zeta;q^{2})_{\infty}}.
		\end{equation}
\end{lemma}
\begin{proof}
	See the paper of Bringmann et al. \cite{BringmannO2009}.
\end{proof}

	Next, we will recall two lemmas from the paper of S. Wilson \cite{Wilson2026}.
	For convenience, we set
	\[
	(1-\zeta^{\pm k}q^n)
	=
	(1-\zeta^kq^n)(1-\zeta^{-k}q^n).
	\]
	\begin{lemma}\label{lema5}
	Let $b_5(n)$ denote the $5$-core partition function. Then
	\[
	C^{(5)}(\zeta;q)
	=
	\prod_{n=1}^{\infty}
	\frac{
		(1-q^n)(1-q^{5n})
		(1-\zeta^{\pm2}q^{5n})(1-\zeta^{\pm4}q^{5n})
	}{
		(1-\zeta^{\pm1}q^n)
	}.
	\]
		is a crank generating function for  $b_5(n)$.
	\end{lemma}
	\begin{proof}
	See [\cite{Wilson2026}, Theorem 1.2.].	
	\end{proof}
\begin{lemma}
	Let $b_t(n)$ denote the $t$-core partition function. Then,
	\begin{equation}\label{eq:C7}
		C^{(7)}(\zeta;q)
		=
		\prod_{n=1}^{\infty}
		\frac{
			(1-q^n)(1-q^{7n})^3
			(1-\zeta^{\pm2}q^{7n})
			(1-\zeta^{\pm4}q^{7n})
		}{
			(1-\zeta^{\pm1}q^n)
		}.
	\end{equation}
	
	\begin{equation}\label{eq:C11}
		\begin{aligned}
			C^{(11)}(\zeta;q)
			&=
			\prod_{n=1}^{\infty}
			\frac{
				(1-q^n)(1-q^{11n})^3
				(1-\zeta^{\pm2}q^{11n})
				(1-\zeta^{\pm4}q^{11n})
			}{
				(1-\zeta^{\pm1}q^n)
			} \\[-1mm]
			&\qquad{}\times
			(1-\zeta^{\pm8}q^{11n})
			(1-\zeta^{\pm10}q^{11n}).
		\end{aligned}
	\end{equation}
	
	\begin{equation}\label{eq:C17}
		\begin{aligned}
			C^{(17)}(\zeta;q)
			&=
			\prod_{n=1}^{\infty}
			\frac{
				\begin{aligned}
					&(1-q^n)(1-q^{17n})^5
					(1-\zeta^{\pm2}q^{17n})
					(1-\zeta^{\pm4}q^{17n})\\
					&\quad{}\times
					(1-\zeta^{\pm8}q^{17n})
					(1-\zeta^{\pm10}q^{17n})
					(1-\zeta^{\pm14}q^{17n})\\
					&\quad{}\times
					(1-\zeta^{\pm16}q^{17n})
				\end{aligned}
			}{
				(1-\zeta^{\pm1}q^n)
			},
		\end{aligned}
	\end{equation}
	
	and
	
	\begin{equation}\label{eq:C19}
		\begin{aligned}
			C^{(19)}(\zeta;q)
			&=
			\prod_{n=1}^{\infty}
			\frac{
				\begin{aligned}
					&(1-q^n)(1-q^{19n})^7
					(1-\zeta^{\pm2}q^{19n})
					(1-\zeta^{\pm4}q^{19n})\\
					&\quad{}\times
					(1-\zeta^{\pm8}q^{19n})
					(1-\zeta^{\pm10}q^{19n})
					(1-\zeta^{\pm14}q^{19n})\\
					&\quad{}\times
					(1-\zeta^{\pm16}q^{19n})
				\end{aligned}
			}{
				(1-\zeta^{\pm1}q^n)
			},
		\end{aligned}
	\end{equation}
	are the crank generating functions for $b_7(n)$, $b_{11}(n)$, $b_{17}(n)$, and $b_{19}(n)$ respectively.
\end{lemma}
\begin{proof}
	For the proof one can see [\cite{Wilson2026}, Theorem 1.3.].	
\end{proof}

	\section{Proof of Theorems}

\begin{prop}\label{prop1} For $\zeta=e^{2\pi iz}$ and $ a\in \mathbb{C}$, we have
	\begin{align}
		\zeta^{am}+\zeta^{-am}
		&=
		2\sum_{s=0}^{\infty}
		a^{2s}m^{2s}
		\frac{(2\pi iz)^{2s}}{(2s)!}.
	\end{align} 
\end{prop}
\begin{proof}[Proof of Proposition \ref{prop1}]
	Since we have
	\begin{align*}
		(a;q)_\infty
		&=
		\prod_{n=0}^{\infty}(1-aq^n).
	\end{align*}
	Taking logarithms, we obtain
	\begin{equation}\label{p1}
		\log (a;q)_\infty
		=
		\sum_{n=0}^{\infty}\log(1-aq^n).
	\end{equation}
	Since
	\begin{equation}\label{p2}
		\log(1-x)
		=
		-\sum_{m=1}^{\infty}\frac{x^m}{m},
		\qquad |x|<1,
	\end{equation}
	substituting $x=aq^n$ gives
	\begin{equation}\label{p3}
		\log(1-aq^n)
		=
		-\sum_{m=1}^{\infty}\frac{(aq^n)^m}{m}.
	\end{equation}
	Hence,
	\begin{equation}\label{p4}
		\log(a;q)_\infty
		=
		-\sum_{n=0}^{\infty}\sum_{m=1}^{\infty}
		\frac{a^m q^{mn}}{m}.
	\end{equation}
	As $|q|<1$, we have
	\begin{equation}\label{p5}
		\log(a;q)_\infty
		=
		-\sum_{m=1}^{\infty}\frac{a^m}{m}
		\sum_{n=0}^{\infty}(q^m)^n \\
		=
		-\sum_{m=1}^{\infty}
		\frac{a^m}{m(1-q^m)}.
	\end{equation}
	Since $\zeta=e^{2\pi iz}$, we have
	\begin{equation}\label{p6}
		\zeta^{am}
		=e^{2\pi iamz}\\
		=\sum_{r=0}^{\infty}
		\frac{(2\pi i\,amz)^r}{r!},
	\end{equation}
	and
	\begin{equation}\label{p7}
		\zeta^{-am}
		=e^{-2\pi iamz}\\
		=\sum_{r=0}^{\infty}
		\frac{(-2\pi i\,amz)^r}{r!}.
	\end{equation}
	Hence,
	\begin{equation}\label{p8}
		\zeta^{am}+\zeta^{-am}
		=
		\sum_{r=0}^{\infty}
		\frac{(2\pi i mz)^r}{r!}
		\left(a^r+(-a)^r\right).
	\end{equation}
	Observe that
	\[
	a^r+(-a)^r=
	\begin{cases}
		0, & \text{if } r \text{ is odd},\\[2mm]
		2a^r, & \text{if } r \text{ is even}.
	\end{cases}
	\]
	Therefore,
	\begin{equation}
		\zeta^{am}+\zeta^{-am}
		=
		\sum_{s=0}^{\infty}
		\frac{(2\pi i mz)^{2s}}{(2s)!}\,2a^{2s}\\
		=
		2\sum_{s=0}^{\infty}
		a^{2s}m^{2s}
		\frac{(2\pi iz)^{2s}}{(2s)!}.
		\label{eq:9}
	\end{equation} 
\end{proof}
\begin{proof}[Proof of Theorem \ref{th1}] We present the proof for the cases $t=5,11,17$. The remaining cases follow analogously. Lemma \eqref{lema5} can be rewritten as
	\begin{equation}\label{th1.1}
		C^{(5)}(\zeta;q)
		=
		\frac{(q;q)_{\infty}
			(q^{5};q^{5})_{\infty}
			(\zeta^{2}q^{5};q^{5})_{\infty}
			(\zeta^{-2}q^{5};q^{5})_{\infty}
			(\zeta^{4}q^{5};q^{5})_{\infty}
			(\zeta^{-4}q^{5};q^{5})_{\infty}}
		{(\zeta q;q)_{\infty}
			(\zeta^{-1}q;q)_{\infty}}.
	\end{equation}
	Using \eqref{eq1.4}, equation \eqref{th1.1} can be rewritten as
	\begin{equation}\label{th1.2}
		C^{(5)}(\zeta;q)
		=
		C(\zeta;q)
		(q^{5};q^{5})_{\infty}
		(\zeta^{2}q^{5};q^{5})_{\infty}
		(\zeta^{-2}q^{5};q^{5})_{\infty}
		(\zeta^{4}q^{5};q^{5})_{\infty}
		(\zeta^{-4}q^{5};q^{5})_{\infty}.
	\end{equation}
	Substituting \ref{eq1.4} into \eqref{th1.2}, we obtain
	\begin{align}\label{th1.3}
		C^{(5)}(\zeta;q)
		&=
		\frac{\sin(\pi z)}{\pi z}
		\frac{(q^{5};q^{5})_{\infty}}
		{(q;q)_{\infty}}
		(\zeta^{2}q^{5};q^{5})_{\infty}
		(\zeta^{-2}q^{5};q^{5})_{\infty}
		(\zeta^{4}q^{5};q^{5})_{\infty}
		(\zeta^{-4}q^{5};q^{5})_{\infty}
		\notag\\
		&\qquad\times
		\exp\!\left(
		2\sum_{s\ge2}
		G_s(\tau)
		\frac{(2\pi iz)^s}{s!}
		\right).
	\end{align}
	Hence,
	\begin{equation}\label{th1.4}
		C^{(5)}(\zeta;q)
		=
		\frac{\sin(\pi z)}{\pi z}
		\frac{(q^{5};q^{5})^5_{\infty}}
		{(q;q)_{\infty}}
		\,g^{(5)}(\zeta;q)
		\exp\left(
		2\sum_{s\ge2}
		G_s(\tau)
		\frac{(2\pi iz)^s}{s!}
		\right),
	\end{equation}
	where
	\begin{equation}\label{th1.5}
		g^{(5)}(\zeta;q)
		=
		\frac{
			(\zeta^{2}q^{5};q^{5})_{\infty}
			(\zeta^{-2}q^{5};q^{5})_{\infty}
			(\zeta^{4}q^{5};q^{5})_{\infty}
			(\zeta^{-4}q^{5};q^{5})_{\infty}
		}{
			(q^{5};q^{5})_{\infty}^{4}
		}.
	\end{equation}
	Now, we compute $\log g^{(5)}(\zeta;q)$. From \eqref{p5}, we have
	\begin{align}
		\log (\zeta^{2}q^{5};q^{5})_{\infty}
		&=
		-\sum_{m=1}^{\infty}
		\frac{\zeta^{2m}q^{5m}}
		{m(1-q^{5m})},
		\label{eq11a}
		\\
		\log (\zeta^{-2}q^{5};q^{5})_{\infty}
		&=
		-\sum_{m=1}^{\infty}
		\frac{\zeta^{-2m}q^{5m}}
		{m(1-q^{5m})},
		\label{eq11b}
		\\
		\log (\zeta^{4}q^{5};q^{5})_{\infty}
		&=
		-\sum_{m=1}^{\infty}
		\frac{\zeta^{4m}q^{5m}}
		{m(1-q^{5m})},
		\label{eq11c}
		\\
		\log (\zeta^{-4}q^{5};q^{5})_{\infty}
		&=
		-\sum_{m=1}^{\infty}
		\frac{\zeta^{-4m}q^{5m}}
		{m(1-q^{5m})}.
		\label{eq11d}
	\end{align}
	Hence,
	\begin{align}
		\log g^{(5)}(\zeta;q)
		={}&
		\log (\zeta^{2}q^{5};q^{5})_{\infty}
		+\log (\zeta^{-2}q^{5};q^{5})_{\infty}
		+\log (\zeta^{4}q^{5};q^{5})_{\infty}
		+\log (\zeta^{-4}q^{5};q^{5})_{\infty}
		\notag\\
		&\qquad
		-4\log (q^{5};q^{5})_{\infty} \notag\\
		={}&
		2\sum_{s=0}^{\infty}
		\left\{-\left(2^{2s}+4^{2s}\right)
		\sum_{m=1}^{\infty}
		m^{2s-1}
		\frac{q^{5m}}
		{1-q^{5m}}\right\}
		\frac{(2\pi iz)^{2s}}{(2s)!}
		+
		4\sum_{m=1}^{\infty}
		\frac{q^{5m}}
		{m(1-q^{5m})}
		\notag\\
		={}&
		2\sum_{s=1}^{\infty}
		A^{(5)}_{2s}(\tau)
		\frac{(2\pi iz)^{2s}}{(2s)!},
		\label{th1.6}
	\end{align}
	where
	\begin{equation}
		A^{(5)}_{2s}(\tau)
		=
		-\left(2^{2s}+4^{2s}\right)
		\sum_{m=1}^{\infty}
		m^{2s-1}
		\frac{q^{5m}}
		{1-q^{5m}},
		\qquad s\ge1.
		\label{th1.7}
	\end{equation}
	Therefore,
	\begin{align}
		C^{(5)}(\zeta;q)
		&=
		\frac{\sin(\pi z)}{\pi z}
		\frac{(q^{5};q^{5})_{\infty}^{5}}
		{(q;q)_{\infty}}
		\exp\left(
		2\sum_{s\ge1}
		A^{(5)}_{2s}(\tau)
		\frac{(2\pi iz)^{2s}}{(2s)!}
		\right)
		\notag\\
		&\qquad\times
		\exp\left(
		2\sum_{s\ge2}
		G_{s}(\tau)
		\frac{(2\pi iz)^{s}}{s!}
		\right)\notag\\
		={}&
		\frac{\sin(\pi z)}{\pi z}
		\frac{(q^{5};q^{5})_{\infty}^{5}}
		{(q;q)_{\infty}}
		\exp\left(
		2\sum_{s\ge1}
		\left\{A^{(5)}_{2s}(\tau)+G_{2s}(\tau)\right\}
		\frac{(2\pi iz)^{2s}}{(2s)!}
		\right)
		\notag\\
		={}&
		\frac{\sin(\pi z)}{\pi z}
		C^{(5)}(1;q)
		\exp\left(
		2\sum_{s\ge1}
		D^{(5)}_{s}(\tau)
		\frac{(2\pi iz)^{s}}{(s)!}
		\right),
		\label{th1.8}
	\end{align}
	where \begin{align*}
		C^{(5)}(1;q)= \frac{(q^{5};q^{5})_{\infty}^{5}}
		{(q;q)_{\infty}}; \notag\\
		D^{(5)}_{s}(\tau)=
		\begin{cases}
			A^{(5)}_s(\tau)+G_s(\tau); & \text{if } s \text{ is even},\\[6pt]
			0; & \text{if } s \text{ is odd}.
		\end{cases}
	\end{align*}
	Hence, by the Lemma \eqref{KANG}
	\[
	\mathcal{C}_{2s}^{(5)}(q)
	=
	\operatorname{Tr}_s
	\left(
	\Psi_T;
	D^{(5)}_2+\frac{B_2}{4},
	D^{(5)}_4+\frac{B_4}{8},
	\ldots,
	D^{(5)}_{2s}+\frac{B_{2s}}{4s}
	\right),
	\]
	\[
	D^{(5)}_{2s}
	=
	-\frac{B_{2s}}{4s}
	+
	\operatorname{Tr}_s
	\left(
	\Psi_V;
	\mathcal{C}_2^{(5)}(q),
	\mathcal{C}_4^{(5)}(q),
	\ldots,
	\mathcal{C}_{2s}^{(5)}(q)
	\right).
	\]
	
	As earlier, we can rewrite \eqref{eq:C11} as 
	\begin{equation}
		\begin{aligned}
			C^{(11)}(\zeta;q)
			=
			\frac{
				(q;q)_{\infty}(q^{11};q^{11})_{\infty}^{3}
				(\zeta^{2}q^{11};q^{11})_{\infty}
				(\zeta^{-2}q^{11};q^{11})_{\infty}
				(\zeta^{4}q^{11};q^{11})_{\infty}
				(\zeta^{-4}q^{11};q^{11})_{\infty}
			}{(\zeta q;q)_{\infty}(\zeta^{-1}q;q)_{\infty}}
			\\[2mm]
			\times
			(\zeta^{8}q^{11};q^{11})_{\infty}
			(\zeta^{-8}q^{11};q^{11})_{\infty}
			(\zeta^{10}q^{11};q^{11})_{\infty}
			(\zeta^{-10}q^{11};q^{11})_{\infty}.
		\end{aligned}
		\label{eq:11corecrank}
	\end{equation}
	Hence,
	\begin{equation}
		\begin{aligned}
			C^{(11)}(\zeta;q)
			&=
			\frac{\sin(\pi z)}{\pi z}
			\frac{(q^{11};q^{11})_{\infty}^{11}}
			{(q;q)_{\infty}}
			\,g^{(11)}(\zeta;q)      \\
			&\qquad\times
			\exp\left(
			2\sum_{s\ge2}
			G_s(\tau)
			\frac{(2\pi iz)^s}{s!}
			\right),
		\end{aligned}
		\label{eq:11gexp}
	\end{equation}
	where
	\begin{equation}
		\begin{aligned}
			g^{(11)}(\zeta;q)
			=
			\frac{
				(\zeta^{2}q^{11};q^{11})_{\infty}
				(\zeta^{-2}q^{11};q^{11})_{\infty}
				(\zeta^{4}q^{11};q^{11})_{\infty}
				(\zeta^{-4}q^{11};q^{11})_{\infty}
			}
			{(q^{11};q^{11})_{\infty}^{8}}
			\\[2mm]
			\times
			(\zeta^{8}q^{11};q^{11})_{\infty}
			(\zeta^{-8}q^{11};q^{11})_{\infty}
			(\zeta^{10}q^{11};q^{11})_{\infty}
			(\zeta^{-10}q^{11};q^{11})_{\infty}.
		\end{aligned}
		\label{eq:11g}
	\end{equation} 
	A similar calculation as earlier leads us to
	\begin{align}
		\log g^{(11)}(\zeta;q)
		={}&
		2\sum_{s=0}^{\infty}
		\left\{
		-\left(
		2^{2s}+4^{2s}+8^{2s}+10^{2s}
		\right)
		\sum_{m=1}^{\infty}
		m^{2s-1}
		\frac{q^{11m}}
		{1-q^{11m}}
		\right\}
		\frac{(2\pi iz)^{2s}}{(2s)!}
		\notag\\
		&\qquad
		+
		8\sum_{m=1}^{\infty}
		\frac{q^{11m}}
		{m(1-q^{11m})}
		\notag\\
		={}&
		2\sum_{s=1}^{\infty}
		A^{(11)}_{2s}(\tau)
		\frac{(2\pi iz)^{2s}}{(2s)!},
		\label{eq:11coreA}
	\end{align}
	where
	\begin{equation}
		A^{(11)}_{2s}(\tau)
		=
		-\left(
		2^{2s}+4^{2s}+8^{2s}+10^{2s}
		\right)
		\sum_{m=1}^{\infty}
		m^{2s-1}
		\frac{q^{11m}}
		{1-q^{11m}},
		\qquad s\ge1.
	\end{equation}
	
	Therefore,
	\begin{align}
		C^{(11)}(\zeta;q)
		&=
		\frac{\sin(\pi z)}{\pi z}
		\frac{(q^{11};q^{11})_{\infty}^{11}}
		{(q;q)_{\infty}}
		\exp\left(
		2\sum_{s\ge1}
		A^{(11)}_{2s}(\tau)
		\frac{(2\pi iz)^{2s}}{(2s)!}
		\right)
		\notag\\
		&\qquad\times
		\exp\left(
		2\sum_{s\ge2}
		G_s(\tau)
		\frac{(2\pi iz)^s}{s!}
		\right)
		\notag\\
		={}&
		\frac{\sin(\pi z)}{\pi z}
		\frac{(q^{11};q^{11})_{\infty}^{11}}
		{(q;q)_{\infty}}
		\exp\left(
		2\sum_{s\ge1}
		\left\{
		A^{(11)}_{2s}(\tau)+G_{2s}(\tau)
		\right\}
		\frac{(2\pi iz)^{2s}}{(2s)!}
		\right)
		\notag\\
		={}&
		\frac{\sin(\pi z)}{\pi z}
		C^{(11)}(1;q)
		\exp\left(
		2\sum_{s\ge1}
		D^{(11)}_{s}(\tau)
		\frac{(2\pi iz)^s}{s!}
		\right),
		\label{eq:11coreD}
	\end{align}
	where
	\begin{align*}
		C^{(11)}(1;q)
		&=
		\frac{(q^{11};q^{11})_{\infty}^{11}}
		{(q;q)_{\infty}},\\
		D^{(11)}_{s}(\tau)
		&=
		\begin{cases}
			A^{(11)}_{s}(\tau)+G_s(\tau), & \text{if } s \text{ is even},\\[6pt]
			0, & \text{if } s \text{ is odd}.
		\end{cases}
	\end{align*}
	Hence, by the Lemma \eqref{KANG}
	\[
	\mathcal{C}_{2s}^{(11)}(q)
	=
	\operatorname{Tr}_s
	\left(
	\Psi_T;
	D^{(11)}_{2}+\frac{B_2}{4},
	D^{(11)}_{4}+\frac{B_4}{8},
	\ldots,
	D^{(11)}_{2s}+\frac{B_{2s}}{4s}
	\right),
	\]
	and
	\[
	D^{(11)}_{2s}
	=
	-\frac{B_{2s}}{4s}
	+
	\operatorname{Tr}_s
	\left(
	\Psi_V;
	\mathcal{C}_2^{(11)}(q),
	\mathcal{C}_4^{(11)}(q),
	\ldots,
	\mathcal{C}_{2s}^{(11)}(q)
	\right).
	\]

	Lastly, we want to proof the theorem for $t=17$.
	\begin{equation}
		\begin{aligned}
			C^{(17)}(\zeta;q)
			=
			\frac{
				(q;q)_{\infty}(q^{17};q^{17})_{\infty}^{5}
				(\zeta^{2}q^{17};q^{17})_{\infty}
				(\zeta^{-2}q^{17};q^{17})_{\infty}
				(\zeta^{4}q^{17};q^{17})_{\infty}
				(\zeta^{-4}q^{17};q^{17})_{\infty}
			}{(\zeta q;q)_{\infty}(\zeta^{-1}q;q)_{\infty}}
			\\[2mm]
			\times
			(\zeta^{8}q^{17};q^{17})_{\infty}
			(\zeta^{-8}q^{17};q^{17})_{\infty}
			(\zeta^{10}q^{17};q^{17})_{\infty}
			(\zeta^{-10}q^{17};q^{17})_{\infty}
			\\[2mm]
			\times
			(\zeta^{14}q^{17};q^{17})_{\infty}
			(\zeta^{-14}q^{17};q^{17})_{\infty}
			(\zeta^{16}q^{17};q^{17})_{\infty}
			(\zeta^{-16}q^{17};q^{17})_{\infty}.
		\end{aligned}
		\label{eq:17corecrank}
	\end{equation}
	
	Hence,
	\begin{equation}
		\begin{aligned}
			C^{(17)}(\zeta;q)
			&=
			\frac{\sin(\pi z)}{\pi z}
			\frac{(q^{17};q^{17})_{\infty}^{17}}
			{(q;q)_{\infty}}
			\,g^{(17)}(\zeta;q)
			\\
			&\qquad\times
			\exp\left(
			2\sum_{s\ge2}
			G_s(\tau)
			\frac{(2\pi iz)^s}{s!}
			\right),
		\end{aligned}
		\label{eq:17gexp}
	\end{equation}
	where
	\begin{equation}
		\begin{aligned}
			g^{(17)}(\zeta;q)
			=
			\frac{
				(\zeta^{2}q^{17};q^{17})_{\infty}
				(\zeta^{-2}q^{17};q^{17})_{\infty}
				(\zeta^{4}q^{17};q^{17})_{\infty}
				(\zeta^{-4}q^{17};q^{17})_{\infty}
			}
			{(q^{17};q^{17})_{\infty}^{12}}
			\\[2mm]
			\times
			(\zeta^{8}q^{17};q^{17})_{\infty}
			(\zeta^{-8}q^{17};q^{17})_{\infty}
			(\zeta^{10}q^{17};q^{17})_{\infty}
			(\zeta^{-10}q^{17};q^{17})_{\infty}
			\\[2mm]
			\times
			(\zeta^{14}q^{17};q^{17})_{\infty}
			(\zeta^{-14}q^{17};q^{17})_{\infty}
			(\zeta^{16}q^{17};q^{17})_{\infty}
			(\zeta^{-16}q^{17};q^{17})_{\infty}.
		\end{aligned}
		\label{eq:17g}
	\end{equation}
	
	A similar calculation as earlier leads to
	\begin{align}
		\log g^{(17)}(\zeta;q)
		={}&
		2\sum_{s=0}^{\infty}
		\Biggl\{
		-\left(
		2^{2s}+4^{2s}+8^{2s}+10^{2s}+14^{2s}+16^{2s}
		\right)
		\sum_{m=1}^{\infty}
		m^{2s-1}
		\frac{q^{17m}}
		{1-q^{17m}}
		\notag\\
		&\qquad\qquad
		+
		6\sum_{m=1}^{\infty}
		\frac{q^{17m}}
		{m(1-q^{17m})}
		\Biggr\}
		\frac{(2\pi iz)^{2s}}{(2s)!}
		\notag\\
		={}&
		2\sum_{s=1}^{\infty}
		A^{(17)}_{2s}(\tau)
		\frac{(2\pi iz)^{2s}}{(2s)!},
		\label{eq:17coreA}
	\end{align}
	where
	\begin{equation}
		A^{(17)}_{2s}(\tau)
		=
		-\left(
		2^{2s}+4^{2s}+8^{2s}+10^{2s}+14^{2s}+16^{2s}
		\right)
		\sum_{m=1}^{\infty}
		m^{2s-1}
		\frac{q^{17m}}
		{1-q^{17m}},
		\qquad s\ge1.
	\end{equation}
	
	Therefore,
	\begin{align}
		C^{(17)}(\zeta;q)
		&=
		\frac{\sin(\pi z)}{\pi z}
		\frac{(q^{17};q^{17})_{\infty}^{17}}
		{(q;q)_{\infty}}
		\exp\left(
		2\sum_{s\ge1}
		A^{(17)}_{2s}(\tau)
		\frac{(2\pi iz)^{2s}}{(2s)!}
		\right)
		\notag\\
		&\qquad\times
		\exp\left(
		2\sum_{s\ge2}
		G_s(\tau)
		\frac{(2\pi iz)^s}{s!}
		\right)
		\notag\\
		={}&
		\frac{\sin(\pi z)}{\pi z}
		\frac{(q^{17};q^{17})_{\infty}^{17}}
		{(q;q)_{\infty}}
		\exp\left(
		2\sum_{s\ge1}
		\left\{
		A^{(17)}_{2s}(\tau)+G_{2s}(\tau)
		\right\}
		\frac{(2\pi iz)^{2s}}{(2s)!}
		\right)
		\notag\\
		={}&
		\frac{\sin(\pi z)}{\pi z}
		C^{(17)}(1;q)
		\exp\left(
		2\sum_{s\ge1}
		D^{(17)}_{s}(\tau)
		\frac{(2\pi iz)^s}{s!}
		\right),
		\label{eq:17coreD}
	\end{align}
	where
	\begin{align*}
		C^{(17)}(1;q)
		&=
		\frac{(q^{17};q^{17})_{\infty}^{17}}
		{(q;q)_{\infty}},\\
		D^{(17)}_{s}(\tau)
		&=
		\begin{cases}
			A^{(17)}_{s}(\tau)+G_s(\tau), & \text{if } s \text{ is even},\\[6pt]
			0, & \text{if } s \text{ is odd}.
		\end{cases}
	\end{align*}
	Therefore, by the Lemma \eqref{KANG}
	\[
	\mathcal{C}_{2s}^{(17)}(q)
	=
	\operatorname{Tr}_s
	\left(
	\Psi_T;
	D^{(17)}_{2}+\frac{B_2}{4},
	D^{(17)}_{4}+\frac{B_4}{8},
	\ldots,
	D^{(17)}_{2s}+\frac{B_{2s}}{4s}
	\right),
	\]
	
	and
	
	\[
	D^{(17)}_{2s}
	=
	-\frac{B_{2s}}{4s}
	+
	\operatorname{Tr}_s
	\left(
	\Psi_V;
	\mathcal{C}_2^{(17)}(q),
	\mathcal{C}_4^{(17)}(q),
	\ldots,
	\mathcal{C}_{2s}^{(17)}(q)
	\right).
	\]
	
\end{proof}
\begin{remark}
	For every even integer $s\ge2$, we have
	\[
	A^{(5)}_s(\tau)
	=
	-\left(2^s+4^s\right)
	\sum_{m=1}^{\infty}
	m^{s-1}
	\frac{q^{5m}}{1-q^{5m}}
	=
	-\left(2^s+4^s\right)
	\left(
	G_s(5\tau)
	+\frac{B_s}{2s}
	\right),
	\]
	where $G_s(\tau)$ denotes the Eisenstein series of weight $s$. Hence,
	\[
	D^{(5)}_s(\tau)
	=
	G_s(\tau)
	-
	\left(2^s
	+4^s\right)
	G_s(5\tau)
	-
	\left(2^s+4^s\right)\frac{B_s}{2s}.\]
	In particular for $s=2$,	$D^{(5)}_2(\tau)$is quasimodular form for $\Gamma_0(5)$. 
\end{remark}

To illustrate Theorem~\eqref{th1}, we give an example below-

Let $t=5$ and $s=2$. The partitions of $2$ are
\[
(2) \quad \text{and} \quad (1^2).
\]
For the trace $\Psi_T$, we have
\begin{align*}
	\Psi_T((2))
	&=4!\frac{2}{4!}
	=2,
	\\
	\Psi_T((1^2))
	&=4!\frac{2^2}{2!(2!)^2}
	=12.
\end{align*}
Thus,
\begin{align*}
	\operatorname{Tr}_2(\Psi_T;F_1,F_2)
	&=2F_2+12F_1^2.
\end{align*}
Taking
\[
F_1=D_2^{(5)}+\frac{B_2}{4},
\qquad
F_2=D_4^{(5)}+\frac{B_4}{8},
\]
the first identity in \eqref{eq:5} gives
\begin{align*}
	\mathcal{C}_4^{(5)}(q)
	&=
	2\left(D_4^{(5)}+\frac{B_4}{8}\right)
	+12\left(D_2^{(5)}+\frac{B_2}{4}\right)^2.
\end{align*}
Since
\[
B_2=\frac{1}{6},
\qquad
B_4=-\frac{1}{30},
\]
we obtain
\begin{align*}
	\mathcal{C}_4^{(5)}(q)
	&=
	2D_4^{(5)}
	-\frac{1}{120}
	+12\left(D_2^{(5)}+\frac{1}{24}\right)^2.
\end{align*}

Similarly, for the trace $\Psi_V$, we have
\begin{align*}
	\Psi_V((2))
	&=\frac{4!}{2}\frac{1}{4!}
	=\frac{1}{2},
	\\
	\Psi_V((1^2))
	&=\frac{4!}{2}(-1)
	\frac{1}{2!(2!)^2}
	=-\frac{3}{2}.
\end{align*}
Hence,
\begin{align*}
	\operatorname{Tr}_2(\Psi_V;F_1,F_2)
	&=\frac{1}{2}F_2-\frac{3}{2}F_1^2.
\end{align*}
Taking
\[
F_1=\mathcal{C}_2^{(5)}(q),
\qquad
F_2=\mathcal{C}_4^{(5)}(q),
\]
the second identity in \eqref{eq:5} yields
\begin{align*}
	D_4^{(5)}
	&=
	-\frac{B_4}{8}
	+\frac{1}{2}\mathcal{C}_4^{(5)}(q)
	-\frac{3}{2}
	\left(\mathcal{C}_2^{(5)}(q)\right)^2.
\end{align*}
Since $B_4=-1/30$, we obtain
\begin{align*}
	D_4^{(5)}
	&=
	\frac{1}{240}
	+\frac{1}{2}\mathcal{C}_4^{(5)}(q)
	-\frac{3}{2}
	\left(\mathcal{C}_2^{(5)}(q)\right)^2.
\end{align*}

\begin{proof}[Proof of Theorem \ref{th2}]
	Recall that the two-variable crank generating function for overpartitions is
	\begin{equation}\label{ovr1.1}
		\overline{C}(\zeta;q)
		=
		\frac{(-q;q)_\infty(q;q)_\infty}
		{(\zeta q;q)_\infty(\zeta^{-1}q;q)_\infty},
	\end{equation}
	where $\zeta=e^{2\pi iz}$.
	
	Using \eqref{eq1.16},  \eqref{ovr1.1} can be written as
	\begin{align*}
		\overline{C}(\zeta;q)
		&=
		\frac{\sin(\pi z)}{\pi z}
		(-q;q)_\infty
		\exp\left(
		2\sum_{s\ge2}
		G_s(\tau)
		\frac{(2\pi iz)^s}{s!}
		\right)
		\notag\\
		&=
		\frac{\sin(\pi z)}{\pi z}
		\,\overline{C}(1;q)\,
		\exp\left(
		2\sum_{s\ge2}
		G_s(\tau)
		\frac{(2\pi iz)^s}{s!}
		\right),
	\end{align*}
	where
	\[
	\overline{C}(1;q)=(-q;q)_\infty.
	\]
	Therefore, by the Lemma \eqref{KANG}, we get
	\[
	\overline{\mathcal{C}}_{2s}(q)
	=
	\operatorname{Tr}_s
	\left(
	\Psi_T;
	\,
	G_2+\frac{B_2}{4},
	\,
	G_4+\frac{B_4}{8},
	\,
	\ldots,
	\,
	G_{2s}+\frac{B_{2s}}{4s}
	\right),
	\]
	and
	\[
	G_{2s}
	=
	-\frac{B_{2s}}{4s}
	+
	\operatorname{Tr}_s
	\left(
	\Psi_V;
	\,
	\overline{\mathcal{C}}_{2}(q),
	\,
	\overline{\mathcal{C}}_{4}(q),
	\,
	\ldots,
	\,
	\overline{\mathcal{C}}_{2s}(q)
	\right).
	\]
\end{proof}
\begin{proof}[Proof of Theorem \ref{th3}]
	We first rewrite the generating function as
	\begin{equation}\label{ovr2.1}
		\overline{C2}(\zeta;q)
		=
		\frac{(-q;q)_{\infty}}{(q;q^2)_{\infty}}
		\frac{(q^2;q^2)_{\infty}}
		{(\zeta q^2;q^2)_{\infty}
			(\zeta^{-1}q^2;q^2)_{\infty}}.
	\end{equation}
	Hence,
	\begin{equation}\label{ovr2.2}
		\overline{C2}(\zeta;q)
		=
		\frac{(-q;q)_{\infty}}{(q;q^2)_{\infty}}
		C(\zeta;q^2).
	\end{equation}
	Consequently,
	\begin{equation}
		\frac{\overline{C2}(\zeta;q)}
		{\overline{C2}(1;q)}
		=
		\frac{C(\zeta;q^2)}
		{C(1;q^2)}.
		\label{eq:barC2_normalized}
	\end{equation}
	Since $\zeta=e^{2\pi iz}$ and $q=e^{2\pi i\tau}$, using the
	exponential representation of the ordinary crank generating function,
	we have
	\begin{equation}
		C(\zeta;q^2)
		=
		\frac{\sin(\pi z)}
		{\pi z(q^2;q^2)_{\infty}}
		\exp\left(
		2\sum_{s\geq2}
		G_s(2\tau)
		\frac{(2\pi iz)^s}{s!}
		\right).
		\label{eq:crank_q2_exp}
	\end{equation}
	Since
	\begin{equation}
		C(1;q^2)
		=
		\frac{1}{(q^2;q^2)_{\infty}},
	\end{equation}
	equations \eqref{eq:barC2_normalized} and
	\eqref{eq:crank_q2_exp} yield
	\begin{equation}
		\frac{\overline{C2}(\zeta;q)}
		{\overline{C2}(1;q)}
		=
		\frac{\sin(\pi z)}{\pi z}
		\exp\left(
		2\sum_{s\geq2}
		G_s(2\tau)
		\frac{(2\pi iz)^s}{s!}
		\right).
		\label{eq:barC2_exponential}
	\end{equation}
	 From lemma \eqref{KANG}, it follows that
	\begin{equation}
		\overline{C2}_{2s}(q)
		=
		\operatorname{Tr}_s
		\left(
		\Psi_T;
		2G_2(2\tau)+\frac{B_2}{4},
		2G_4(2\tau)+\frac{B_4}{8},
		\ldots,
		2G_{2s}(2\tau)+\frac{B_{2s}}{4s}
		\right),
		\label{eq:barC2_trace}
	\end{equation}
	and
	\begin{equation}
		G_{2s}(2\tau)
		=
		-\frac{B_{2s}}{4s}
		+
		\operatorname{Tr}_s
		\left(
		\Psi_V;
		\overline{C2}_{2}(q),
		\overline{C2}_{4}(q),
		\ldots,
		\overline{C2}_{2s}(q)
		\right).
		\label{eq:barC2_inverse}
	\end{equation}
	This completes the proof.
\end{proof}
	
\section{Applications}

In this section, we illustrate the preceding Bell polynomial
identities through two partition-theoretic applications. We first
consider $t$-core partition numbers and subsequently treat
overpartition numbers, obtaining explicit Bell polynomial
representations together with their inverse formulas.
	\subsection{$t$-core partition numbers and Bell polynomials}
	Taking logarithmic derivatives in\eqref{eq1.2}, we obtain
	\begin{align*}
		q\frac{d}{dq}
		\log
		\left(
		\sum_{n=0}^{\infty}
		b_t(n)q^n
		\right)
		&=
		t\,
		q\frac{d}{dq}
		\log(q^t;q^t)_\infty
		-
		q\frac{d}{dq}
		\log(q;q)_\infty\\
		&=
		\sum_{n=1}^{\infty}
		\left(
		\sigma_1(n)
		-
		t^2
		\sigma_1^{(t)}(n)
		\right)
		q^n,
	\end{align*}
	where
	\[
	\sigma_1^{(t)}(n)
	=
	\begin{cases}
		\sigma_1(n/t),&\text{if }t\mid n,\\
		0,&\text{otherwise}.
	\end{cases}
	\]
	
	Define
	\[
	a^{(t)}(n)
	=
	\sigma_1(n)
	-
	t^2
	\sigma_1^{(t)}(n).
	\]
	
	Hence,
	\[
	q\frac{d}{dq}
	\log
	\left(
	\sum_{n=0}^{\infty}
	b_t(n)q^n
	\right)
	=
	\sum_{n=1}^{\infty}
	a^{(t)}(n)q^n.
	\]
	
	Integrating both sides gives
	\[
	\sum_{n=0}^{\infty}
	b_t(n)q^n
	=
	\exp
	\left(
	\sum_{n=1}^{\infty}
	\frac{a^{(t)}(n)}{n}q^n
	\right)
	=
	\exp
	\left(
	\sum_{n=1}^{\infty}
	(n-1)!a^{(t)}(n)
	\frac{q^n}{n!}
	\right).
	\]
	
	By the exponential generating function of the complete Bell polynomials,
	\[
	\exp
	\left(
	\sum_{n=1}^{\infty}
	x_n\frac{u^n}{n!}
	\right)
	=
	\sum_{n=0}^{\infty}
	B_n(x_1,\ldots,x_n)
	\frac{u^n}{n!},
	\]
	we obtain
	\[
	\sum_{n=0}^{\infty}
	b_t(n)q^n
	=
	\sum_{n=0}^{\infty}
	\frac{
		B_n
		\!\left(
		0!a^{(t)}(1),
		1!a^{(t)}(2),
		\ldots,
		(n-1)!a^{(t)}(n)
		\right)
	}{n!}
	q^n.
	\]
	
	Comparing coefficients of $q^n$, we obtain
	\[
	b_t(n)
	=
	\frac{1}{n!}
	B_n
	\!\left(
	0!a^{(t)}(1),
	1!a^{(t)}(2),
	\ldots,
	(n-1)!a^{(t)}(n)
	\right).
	\]
	
	Equivalently,
	\[
	b_t(n)
	=
	\sum_{\beta=(1^{m_1},\ldots,n^{m_n})\vdash n}
	\prod_{j=1}^{n}
	\frac{1}{m_j!}
	\left(
	\frac{a^{(t)}(j)}{j}
	\right)^{m_j}.
	\]
	
	Conversely, by the Bell polynomial inversion formula,
	\[
	a^{(t)}(n)
	=
	n
	\sum_{\beta\vdash n}
	\mu(\beta)
	\prod_{j=1}^{n}
	\frac{\left(b_t(j)\right)^{m_j}}{m_j!},
	\]
	where
	\[
	\mu(\beta)
	=
	(-1)^{\ell(\beta)-1}
	(\ell(\beta)-1)!.
	\]
	
	For example taking $t=5$ and $n=3$, since $5\nmid3$, we have
	\[
	a^{(5)}(3)
	=
	\sigma_1(3)
	-
	25\sigma_1^{(5)}(3)
	=
	\sigma_1(3)
	=
	4.
	\]
	On the other hand, since
	\[
	b_5(1)=1,\qquad b_5(2)=2,\qquad b_5(3)=3,
	\]
	the Bell polynomial inversion formula gives
	\[
	\begin{aligned}
		a^{(5)}(3)
		&=
		3\left[
		b_5(3)
		-b_5(2)b_5(1)
		+\frac{1}{3}b_5(1)^3
		\right]\\
		&=
		3\left[
		3-2(1)+\frac{1}{3}(1)^3
		\right]\\
		&=4.
	\end{aligned}
	\]
	Thus,
	\[
	a^{(5)}(3)=4=\sigma_1(3).
	\]
	\subsection{Overpartition numbers and Bell polynomials}
	
	Let $\overline{p}(n)$ denote the number of overpartitions of $n$. Their
	generating function is
	\begin{equation}
		\sum_{n=0}^{\infty}\overline{p}(n)q^n
		=
		\frac{(-q;q)_\infty}{(q;q)_\infty}
		=
		\frac{(q^2;q^2)_\infty}{(q;q)_\infty^2}.
		\label{eq:overGF}
	\end{equation}
	
	Taking logarithmic derivatives, we obtain
	\begin{align*}
		q\frac{d}{dq}
		\log
		\left(
		\sum_{n=0}^{\infty}
		\overline{p}(n)q^n
		\right)
		&=
		q\frac{d}{dq}\log(q^2;q^2)_\infty
		-
		2q\frac{d}{dq}\log(q;q)_\infty \\
		&=
		\sum_{n=1}^{\infty}
		\left(
		2\sigma_1(n)
		-
		2\bar{\sigma_1}(n)\right)q^n.
	\end{align*}
	
	For consistency, define
	\[
	\bar{\sigma_1}(n)
	=
	\begin{cases}
		\sigma_1(n/2), & 2\mid n,\\
		0, & 2\nmid n.
	\end{cases}
	\]
	Then set
	\[
	\alpha(n)
	=
	2\sigma_1(n)-2\bar{\sigma_1}(n).
	\]
	Consequently,
	\[
	q\frac{d}{dq}
	\log
	\left(
	\sum_{n=0}^{\infty}
	\overline{p}(n)q^n
	\right)
	=
	\sum_{n=1}^{\infty}\alpha(n)q^n.
	\]
	
	Integrating both sides and using
	$\overline{p}(0)=1$, we obtain
	\[
	\sum_{n=0}^{\infty}\overline{p}(n)q^n
	=
	\exp\left(
	\sum_{n=1}^{\infty}\frac{\alpha(n)}{n}q^n
	\right)
	=
	\exp\left(
	\sum_{n=1}^{\infty}
	(n-1)!\alpha(n)\frac{q^n}{n!}
	\right).
	\]
	
	By the exponential generating function of the complete Bell
	polynomials,
	\[
	\exp\left(
	\sum_{n=1}^{\infty}x_n\frac{u^n}{n!}
	\right)
	=
	\sum_{n=0}^{\infty}
	B_n(x_1,\ldots,x_n)\frac{u^n}{n!},
	\]
	we obtain
	\[
	\sum_{n=0}^{\infty}\overline{p}(n)q^n
	=
	\sum_{n=0}^{\infty}
	\frac{
		B_n\left(
		0!\alpha(1),1!\alpha(2),\ldots,(n-1)!\alpha(n)
		\right)
	}{n!}q^n.
	\]
	
	Comparing coefficients of $q^n$ gives
	\[
	\overline{p}(n)
	=
	\frac{1}{n!}
	B_n\left(
	0!\alpha(1),1!\alpha(2),\ldots,(n-1)!\alpha(n)
	\right).
	\]
	
	Equivalently,
	\[
	\overline{p}(n)
	=
	\sum_{\beta=(1^{m_1},\ldots,n^{m_n})\vdash n}
	\prod_{j=1}^{n}
	\frac{1}{m_j!}
	\left(
	\frac{\alpha(j)}{j}
	\right)^{m_j}.
	\]
	
	Conversely, by the Bell polynomial inversion formula,
	\[
	\alpha(n)
	=
	n
	\sum_{\beta\vdash n}
	\mu(\beta)
	\prod_{j=1}^{n}
	\frac{\left(\overline{p}(j)\right)^{m_j}}{m_j!},
	\]
	where
	\[
	\mu(\beta)
	=
	(-1)^{\ell(\beta)-1}
	(\ell(\beta)-1)!.
	\]
	As an example for $n=5$, we have
	\[
	\alpha(5) = 2\sigma_1(5) - 2\bar\sigma_1(5) = 12.
	\]. We use the overpartition numbers
	\[
	\overline{p}(0)=1,\quad \overline{p}(1)=2,\quad \overline{p}(2)=4,\quad \overline{p}(3)=8,\quad \overline{p}(4)=14,\quad \overline{p}(5)=24,
	\]
	together with
	\[
	\alpha(5) = 5\sum_{\beta \vdash 5} \mu(\beta) \prod_{j=1}^{5} \frac{\overline{p}(j)^{m_j}}{m_j!},
	\qquad
	\mu(\beta) = (-1)^{\ell(\beta)-1}(\ell(\beta)-1)!.
	\]
	Now,
	\begin{center}
		\renewcommand{\arraystretch}{1.6}
		\begin{tabular}{@{}c c c c c@{}}
			\hline
			$\beta \vdash 5$ & $\ell(\beta)$ & $\mu(\beta)$ & 
			$\displaystyle\prod_{j=1}^{5} \frac{\overline{p}(j)^{m_j}}{m_j!}$ & 
			$\displaystyle \mu(\beta)\prod_{j=1}^{5} \frac{\overline{p}(j)^{m_j}}{m_j!}$ \\
			\hline
			$5^1$      & $1$ & $1$   & $\overline{p}(5) = 24$ 
			& $24$ \\
			$4^1 1^1$  & $2$ & $-1$  & $\overline{p}(4)\,\overline{p}(1) = 14 \cdot 2 = 28$ 
			& $-28$ \\
			$3^1 2^1$  & $2$ & $-1$  & $\overline{p}(3)\,\overline{p}(2) = 8 \cdot 4 = 32$ 
			& $-32$ \\
			$3^1 1^2$  & $3$ & $2$   & $\overline{p}(3)\cdot\dfrac{\overline{p}(1)^2}{2!} = 8 \cdot 2 = 16$ 
			& $32$ \\
			$2^2 1^1$  & $3$ & $2$   & $\dfrac{\overline{p}(2)^2}{2!}\cdot \overline{p}(1) = 8 \cdot 2 = 16$ 
			& $32$ \\
			$2^1 1^3$  & $4$ & $-6$  & $\overline{p}(2)\cdot\dfrac{\overline{p}(1)^3}{3!} = 4 \cdot \dfrac{8}{6} = \dfrac{16}{3}$ 
			& $-32$ \\
			$1^5$      & $5$ & $24$  & $\dfrac{\overline{p}(1)^5}{5!} = \dfrac{32}{120} = \dfrac{4}{15}$ 
			& $\dfrac{32}{5}$ \\
			\hline
		\end{tabular}
	\end{center}
	
	Summing the last column:
	\[
	24 - 28 - 32 + 32 + 32 - 32 + \frac{32}{5} = -4 + \frac{32}{5} = \frac{12}{5}.
	\]
	
	Hence
	\[
	\alpha(5) = 5 \cdot \frac{12}{5} = 12,
	\]
	which agrees with the direct evaluation $\alpha(5) = 2\sigma_1(5) - 2\bar\sigma_1(5) = 12$.
	
	\noindent{\bf Data availability statement:} There is no data associated to our manuscript.
	

\end{document}